\documentclass[10pt]{amsart}
\usepackage[cp1251]{inputenc}
\usepackage[english,russian]{babel}
\usepackage{amsmath}
\usepackage{cite}
\usepackage{amssymb}
\usepackage{amsfonts}
\usepackage{graphicx}
\usepackage{float}

\newtheorem{theorem}{Theorem}
\newtheorem{definition}{Definition}
\newtheorem{corollary}{Corollary}

\begin{document}
\renewcommand{\refname}{References}
\renewcommand{\proofname}{Proof.}
\renewcommand{\figurename}{Fig.}

\thispagestyle{empty}

\title[Сentral orders in simple superalgebras]{Сentral orders in simple finite dimensional superalgebras}
\author{{A.S. Panasenko}}%
\address{Alexander Sergeevich Panasenko
\newline\hphantom{iii} Sobolev Institute of Mathematics,
\newline\hphantom{iii} pr. Koptyuga, 4,
\newline\hphantom{iii} 630090, Novosibirsk, Russia.
\newline\hphantom{iii} Novosibirsk State University,
\newline\hphantom{iii} Universitetskiy pr., 1,
\newline\hphantom{iii} 630090, Novosibirsk, Russia}
\email{a.panasenko@g.nsu.ru}%
\thanks{\sc Panasenko, A.S.,
Сentral orders in simple finite dimensional superalgebras}
\thanks{\copyright \ 2020 Panasenko A.S}

\vspace{1cm}
\maketitle {\small
\begin{quote}
\noindent{\sc Abstract. } The well-known Formanek's module finiteness theorem states that every unital prime PI-algebra (i.e. a central order in a matrix algebra by Posner's theorem) embeds into a finitely generated module over its center. An analogue of this theorem for alternative and Jordan algebras was earlier proved by V.N.~Zhelyabin and the author. In this paper we discuss this problem for associative, classical Jordan and some alternative superalgebras. 
\medskip

 \end{quote}
}

\section{Introduction}

The study of simple and prime (super)algebras in different useful varieties is an important part of the structure theory of these varieties. The prime algebras with polynomial identities are especially interesting. The first point of this theory is Posner's Theorem \cite{Posner}. It states that all associative prime PI-rings coincide with central orders in matrix algebras over finite dimensional division algebras. After this result, the similar theory was constructed in other varieties of nearly associative algebras. M.~Slater \cite{Slater} proved that an alternative prime non-degenerate algebra is a central order in a matrix algebra or in the Cayley--Dickson algebra. Later, E.I.~Zel'manov \cite{Zelm} obtained a description of prime non-degenerate Jordan PI-algebras: they are central orders in finite dimensional central simple Jordan algebras or in the algebra of a non-degenerate symmetric bilinear form. 

For alternative superalgebras the similar result was obtained by I.P.~Shestakov \cite{Shest}. He proved that every alternative non-associative prime superalgebra under some restriction on the even part (for example, if it is non-degenerate) is a central order in a finite dimensional alternative superalgebra.

These results motivate one to study central orders in simple (super)algebras more carefully. For example, E.~Formanek in \cite{Formanek} proved that prime PI-algebra embeds into a finitely generated module over its center. It is important to notice that the proof of this result is essentially based on the Posner's Theorem.

In paper \cite{Panasenko} the author proved that an alternative prime non-degenerate algebra embeds into a finitely generated module over its center. In \cite{Zhel1}, V.N.~Zhelyabin and the author showed that a central order in a finite dimensional simple Jordan algebra embeds into a finitely generated module over its center.

It should be noticed that the study of Jordan superalgebras have been continued. For example, in works \cite{Zhel2,Zhel3,Zhel4} V.N.~Zhelyabin and A.S.~Zakharov studied simple Jordan superalgebras with an associative even part. In works \cite{Gomez1,Gomez2,Gomez3,Gomez4} F.A.~Gomez Gonzalez proved analogues of the Wedderburn Principal Theorem for the cases, when the algebra $A/N$ is one of the classical simple Jordan superalgebras ($N$ is the solvable radical of a superalgebra $A$).

In this paper, we study central orders in simple finite dimensional associative, alternative and Jordan superalgebras. We prove an analogue of Formanek's Theorem for associative superalgebras and (with some restrictions) for alternative and Jordan superalgebras.

\bigskip

\section{Central orders in simple finite dimensional associative superalgebras}

The definition of $\mathcal{M}$-superalgebra in an arbitrary variety $\mathcal{M}$ of nonassociative algebras may be found in a lot of papers (for example, in \cite{RacineZelm}).

Let us remind some notations. If $A$ is a (super)algebra and $x,y,z\in A$ then $$(x,y,z)=(xy)z-x(yz), \qquad [x,y]=xy-yx.$$




\begin{definition}
	If $A$ is a (super)algebra then a subalgebra $$Z(A)=\{x\in A\mid (x,a,b)=(a,x,b)=(a,b,x)=[a,x]=0 \quad\forall a,b\in A\}$$
	is called a \textbf{center} of a (super)algebra $A$.
\end{definition}




\noindent\textbf{Example 1.} Let $B=M_{n+m}(A)$ be the $(n+m)\times (n+m)$ matrix algebra with elements from an associative $k$-algebra $A$. Then $B$ is a $Z_2$-graded algebra with the following grading for every $X\in B$:
$$X=X_0+X_1, \qquad X_0=\begin{pmatrix}
	D_1 & 0 \\
	0 & D_2 
	\end{pmatrix}, \qquad
X_1=\begin{pmatrix}
	0 & D_3 \\
	D_4 & 0 
	\end{pmatrix},
$$
where $D_1\in M_{n}(A)$, $D_2\in M_m(A)$, $D_3\in M_{n,m}(A)$, $D_4\in M_{m,n}(A)$.


\medskip

It is easy to notice that in the example above $M_0\simeq M_p(A)\oplus M_q(A)$, where $\oplus$ denotes the direct sum of ideals of $M_0$.

Let $B$ be an algebra. We will use a notation $C_B(u)=\{x\in B\mid xu=ux\}$ for the centralizer of an element $u$ in $B$. Similarly, $S_B(u)=\{x\in B\mid xu=-ux\}$.

\begin{definition}
	A superalgebra $A$ is called \textbf{central} over a field $k$, if $Z(A)_0$ is isomorphic to $k$.
\end{definition}

M.L.~Racine in \cite{Racine} proved that an associative Artinian simple superalgebra is an algebra of linear maps of some vector superspace over some division superalgebra.  In the same paper all division superalgebras were described. Let us give a part of this description for finite dimensional algebras:

\noindent\textbf{Theorem} (M.L. Racine, \cite{Racine}).
	{\it If $D=D_0+D_1$ is a central finite dimensional division superalgebra over a field $k$, then one and only one of the following statements holds (everywhere below $\varepsilon$ stands for some division algebra over a field $k$):
	
	1) $D=D_0=\varepsilon$, $D_1=0$;
	
	2) $D=\varepsilon\otimes_k k[u]$, $u^2=\lambda \in k$, $\lambda\neq 0$, $D_0=\varepsilon\otimes_k k1$, $D_1=\varepsilon\otimes_k ku$;
	
	3) $D=\varepsilon$, $D_0=C_{\varepsilon}(u)$, where $k[u]\subset \varepsilon$ is a quadratic extension of the field $k$;
	
	4) $D=M_2(\varepsilon)=\varepsilon\otimes_k M_2(k)$, $D_0=\varepsilon\otimes_k k[u]$, $D_1=\varepsilon\otimes_k k[u]w$, where $u=\begin{pmatrix}
	0 & 1 \\
	\lambda & p 
	\end{pmatrix}$, $w=\begin{pmatrix}
	1 & 0 \\
	p & -1
	\end{pmatrix}$, and $k[u]$ is not contained in $\varepsilon$. Here $p=0$ if $\mathrm{char}(k)\neq 2$ and $p=1$ if $\mathrm{char}(k)=2$.}

We will prove the following theorem by using this classification.

\begin{theorem}
	Let $B=B_0 + B_1$ be a unital associative superalgebra, $Z=Z(B)_0$ does not have zero divisors of $B$, $A=Z^{-1}B$ is the central closure of a superalgebra $B$. If the superalgebra $A$ is simple and finite dimensional, then $B$ embeds into a free finitely generated $Z$-module.
\end{theorem}
\begin{proof}
	 In \cite{Formanek} E. Formanek proved this theorem for the case $B_1=A_1=0$. By Racine's Theorem \cite{Racine} $A\simeq End_D(V)$ (algebra isomorphism), where $D$ is a division superalgebra and $V$ is a superspace over $D$. We will use the following notations: $k=Z^{-1}Z$, $k_1$ means the odd part of the center of the superalgebra $A$. 

1) Case $D=D_0=\varepsilon$. Then $End_D(V)\simeq M_{p+q}(D)$. $Z(M_{p+q}(D))\subset A_0$ implies $Z(B)\subset B_0$ and $A\simeq Z(B)^{-1}B$. Since $M_{p+q}(D)$ is a simple algebra, we have our statement by Formanek' Theorem

2) Case $D=\varepsilon\otimes_k k[u]$, $u^2=\lambda\in k$ and $\lambda\neq 0$. Then $A\simeq M_n(D_0)+M_n(D_1)$. 

Firstly, we will consider the case $\varepsilon=k$, i.e. $D\simeq k[u]$. Let $a\in B$. Then $a=\sum\limits_{i,j=1}^{n} (\alpha_{ij}+\beta_{ij}u)e_{ij}$. There exists $z\in Z$, $z\neq 0$, such that $e_{ij}=f_{ij}/z$, $u=v/z, \lambda=\mu/z$, where $f_{ij}, v\in B$, $\mu, z\in Z$. Note that $$\sum\limits_{t=1}^{n}f_{ti}af_{jt}=z^2(\alpha_{ij}+\beta_{ij}u)\in B\cap Z(A)=Z(B),$$ that imply $\alpha_{ij}z^2\in Z$. But $$z^2 (\alpha_{ij}+\beta_{ij}u) v = z^2(\alpha_{ij}v+\beta_{ij}\mu)\in B\cap Z(A)=Z(B),$$ hence $\beta_{ij}\mu z^2\in Z$. It means that $B$ embeds into a free finitely generated $Z$-module $\sum Ze_{ij} + \sum Zu e_{ij}$. 

Now, let $\varepsilon$ be an arbitrary  finite dimensional division algebra. Then, according to (\cite{Herstein}, p.105), $\varepsilon$ contains a maximal subfield $P=k[b]$, $m=\dim_k P$. Moreover, we can assume that $b\in B_0$. Note, that 
\begin{multline*}Z^{-1}(B\otimes_Z Z[b])\simeq M_n(D)\otimes_Z Z^{-1}Z[b]\simeq M_n(D)\otimes_Z k[b]\simeq M_n(D)\otimes_k k[b]\simeq\\ \simeq M_n(\varepsilon\otimes_k k[u])\otimes_k P\simeq k[u]\otimes_k M_n(\varepsilon)\otimes_k P\end{multline*}
By (\cite{Herstein}, p.96) we have an isomorphism $M_n(\varepsilon)\otimes_k P\simeq M_{mn}(P)$. So, 
\begin{multline*} k[u]\otimes_k M_n(\varepsilon)\otimes_k P\simeq M_{mn}(P)\otimes_k k[u]\simeq\\ \simeq M_{mn}(P)\otimes_P k[u]\simeq M_{mn}(P)\otimes_P P[u]\simeq M_{mn}(P\otimes_P P[u])\simeq M_{mn}(P[u]).\end{multline*}
It means that $B\otimes_Z Z[b]$ satisfies the conditions of case 1 and hence $B\otimes_Z Z[b]$ embeds into a free finitely generated $Z[b]$-module. Since $Z[b]$ is finitely generated over $Z$, $B\otimes_Z Z[b]$ embeds into a finitely generated $Z$-module. It remains to notice that $B\subset B\otimes_Z Z[b]$. 

3) Case $D=\varepsilon$, $D_0=C_\varepsilon(u)$, $k[u]\subset\varepsilon$. Let us prove that $Z(\varepsilon)=k$.  It is enough to show that  $Z(\varepsilon)_1=0$. If $0\neq x\in Z(\varepsilon)_1$, then $ux=xu$, that implies $x\in D_0$, a contradiction. Now we can use the Formanek's Theorem for non-graded algebras.

4) Case $D=M_2(\varepsilon)=\varepsilon\otimes_k M_2(k)$, $D_0=\varepsilon\otimes_k k[u]$, $D_1=\varepsilon\otimes k[u]w$, where $w\in M_2(k)$, and $k[u]$ is not contained in $\varepsilon$.

Let us consider the case $\varepsilon=k$. In this case $D=M_2(k)$, hence $Z(A)_1=0$. The algebra $M_n(M_2(k))$ is simple, so we can use the Formanek's Theorem.

Now, let $\varepsilon$ be an arbitrary division algebra. It is obvious that $Z(M_n(M_2(\varepsilon)))\simeq Z(\varepsilon)\simeq k$ like in case 3. Hence, we can use the Formanek's Theorem. 
\end{proof}

\section{Central orders in simple finite dimensional alternative superalgebras}


In \cite{ZelmShest} E.I.~Zel'manov and I.P.~Shestakov proved that all simple alternative superalgebras are either associative or trivial, if a field characteristic is not 2 or 3. Let us state examples of simple alternative superalgebras over fields with a characteristic equal to 2 or 3.

\medskip

\noindent\textbf{Example 2.} a) $\mathrm{Char}(k)=2$. Let $\mathbf{O}=\mathbf{H}+v\mathbf{H}$ be the Calyley--Dickson algebra with a natural $Z_2$-grading, $\mathbf{H}$ are quaternions over a field $k$. Then $\mathbf{O}=\mathbf{H}+v\mathbf{H}$ is a simple alternative superalgebra.

b) $\mathrm{Char}(k)=2$. We need to apply the Cayley--Dickson process to algebra $\mathbf{O}$. We have $k[u]=k+ku$, $u^2=\alpha\neq 0\in k$. Then $\mathbf{O}[u]=k[u]\otimes_k\mathbf{O}=\mathbf{O}+\mathbf{O}u$ is a simple alternative superalgebra.

c) $\mathrm{Char}(k)=3$. Let $A=k1$ be one-dimensional space and let $M=kx+ky$ be a two-dimensional space over a field $k$. We will use a notation $\mathbf{B(1,2)}=A+M$ for the commutative superalgebra over $k$, where $1$ is a unit in $\mathbf{B(1,2)}$ and $xy=-yx=1$. Then $\mathbf{B(1,2)}$ is a simple alternative superalgebra.

d) $\mathrm{Char}(k)=3$. Let $A=M_2(k)$ be a $2\times 2$ matrix algebra over a field $k$, $M=km_1+km_2$ be a two-dimensional space over $k$. We will use a notation $\mathbf{B(4,2)}=A+M$ for $Z_2$-graded algebra over $k$, where $e_{ij}m_k=\delta_{ij}m_j$, $m_1m_2=e_{11}$, $m_2m_1=-e_{22}$, $m_1^2=-e_{21}$, $m_2^2=e_{12}$ and $ma=\overline{a}m$ for $a\in A$. Here $\overline{a}$ is the symplectic involution. Then $\mathbf{B(4,2)}$ is a simple alternative superalgebra.

e) $\mathrm{Char}(k)=3$. Let $\Gamma$ be an associative and commutative $\partial$-simple algebra, $\partial$ is a derivation on $\Gamma,$ $\gamma\in\Gamma$ and $\overline{\Gamma}$ is an isomorpic copy of linear space $\Gamma.$ Consider $B(\Gamma, \partial, \gamma)=\Gamma+\overline{\Gamma}$ is a direct sum of linear spaces with multiplication below:
$$x\cdot y = xy,$$
$$x\cdot\overline{y}=\overline{x}\cdot y = \overline{xy},$$
$$\overline{x}\cdot\overline{y}=(\gamma xy + 2\partial(x)y + x\partial(y)),$$
where $x,y\in\Gamma$, $xy$ is the product in $\Gamma.$ Then $B(\Gamma,\partial,\gamma)$ is a simple alternative superalgebra.

In paper \cite{Shest} I.P.~Shestakov classified all simple alternative superalgebras over a field of characteristic 2 or 3.

\medskip 

\noindent\textbf{Theorem} (I.P. Shestakov, \cite{Shest}). {\it Let $B=A+M$ be a simple alternative non-associative central superalgebra over a field $k$. If $M\neq 0$, then $B$ is isomorphic to one of five algebras from Example 2.}

Now, using this classification, we prove the following

\begin{theorem}
	 Let $B$ be a unital alternative non-associative superalgebra and $Z=Z(B)_{0}$. If $Z^{-1}B$ is a finite dimensional central simple superalgebra, then either $B$ embeds (as a $Z$-submodule) into a free finitely generated $Z$-module or $Z^{-1}B$ is isomorphic to $B(\Gamma,\partial,\gamma)$.
	 \end{theorem}

\medskip

\begin{proof}
	 Let $B=A+M$ be a unital alternative non-associative superalgebra and $Z=Z(B)_{0}$. We may define  $k=Z^{-1}Z,$ $\overline{B}=Z^{-1}B,$ $\overline{A}=\overline{B}_{0},$ $\overline{M}=\overline{B}_{1}$. 
By Shestakov's theorem we have to consider one of the following cases:

\medskip

1. $char F = 2,$ $\overline{B}=\mathbf{O}=\mathbf{H}+v\mathbf{H}$ is the Cayley--Dickson algebra over $k$ with natural grading, generated by the Cayley--Dickson process, $\mathbf{H}$ is the quaternion algebra. There exist $\mu\in k,\beta,\gamma\in (k\setminus 0)$ such that the multiplication table of $\overline{B}$ in a certain basis $e_{0}, e_{1}, ..., e_{7}$, $e_0=1$, has the following form (see., for example, \cite{Zhevl}, \S2.2): 

\begin{table}[H]
	\begin{center}
		\begin{tabular}{|c|c|c|c|c|c|c|c|}
			\hline
			& $e_{1}$ & $e_{2}$ & $e_{3}$ & $e_{4}$ & $e_{5}$ & $e_{6}$ & $e_{7}$\\
			\hline
			$e_{1}$ & $e_{1}+\mu$ & $e_{2}+e_{3}$ & $\mu e_{2}$ & $e_{4}+e_{5}$ & $\mu e_{4}$ & $e_{7}$ & $e_{7}+\mu e_{6}$\\
			\hline
			$e_{2}$ & $e_{3}$ & $\beta$ & $\beta e_{1}$ & $e_{6}$ & $e_{7}$ & $\beta e_{4}$ & $\beta e_{5}$\\
			\hline
			$e_{3}$ & $e_{3}+\mu e_{2}$ & $\beta + \beta e_{1}$ & $\mu\beta$ & $e_{7}$ & $\alpha e_{6}+e_{7}$ & $\beta (e_{5}+e_{4})$ & $\mu\beta e_{4}$\\
			\hline
			$e_{4}$ & $e_{5}$ & $e_{6}$ & $e_{7}$ & $\gamma$ & $\gamma e_{1}$ & $\gamma e_{2}$ & $\gamma e_{3}$\\
			\hline
			$e_{5}$ & $e_{5}+\mu e_{4}$ & $e_{7}$ & $e_{7}+\mu e_{6}$ & $\gamma + \gamma e_{1}$ & $\mu\gamma$ & $\gamma (e_{3}+e_{2})$ & $\mu\gamma e_{2}$\\
			\hline
			$e_{6}$ & $e_{6}+e_{7}$ & $\beta e_{4}$ & $\beta (e_{5}+e_{4})$ & $\gamma e_{2}$ & $\gamma (e_{2}+e_{3})$ & $\beta\gamma$ & $\beta\gamma (e_{1}-1)$\\
			\hline
			$e_{7}$ & $\mu e_{6}$ & $\beta e_{5}$ & $\beta\gamma e_{4}$ & $\gamma e_{3}$ & $\mu\gamma e_{2}$ & $\beta\gamma e_{1}$ & $\mu\beta\gamma$\\
			\hline
		\end{tabular}
	\end{center}
\end{table} 

There exists $z\in Z$, $z\neq 0$, such that $e_{i}=\frac{f_{i}}{z},
\mu=\frac{\nu}{z}, \beta=\frac{\lambda}{z},
\gamma=\frac{\delta}{z},$ where  $f_{i}\in \overline{B}$, $\delta,\lambda,\in
Z\setminus\{0\}$, $\nu\in Z$. Let $a\in B$, $a=\sum\limits_{i=0}^{7}\alpha_{i}e_{i}$. We will use a notation $a\circ b=ab+ba$. We have two cases:

\medskip

1a. $\mu \neq 0.$ Then the following relations hold:
$$\nu z\alpha_{1}=(az-a\circ f_{1})(f_{1}-z),$$
$$z\nu\lambda\alpha_{2}=\nu(a\circ f_{3})z -(az-a\circ f_{1})(f_{1}-z)f_{3}, $$
$$z\nu\lambda\alpha_{3}=\nu(a\circ f_{2})z-(az-a\circ f_{1})(f_{1}-z)f_{2}, $$
$$z\nu\delta\alpha_{4}=\nu(a\circ f_{5})z-(az-a\circ f_{1})(f_{1}-z)f_{5}, $$
$$z\nu\delta\alpha_{5}=\nu(a\circ f_{4})z-(az-a\circ f_{1})(f_{1}-z)f_{4}, $$
$$\nu\lambda\delta\alpha_{6}=\nu(a\circ f_{7})z-(az-a\circ f_{1})(f_{1}-z)f_{7}, $$
$$\nu\lambda\delta\alpha_{7}=\nu(a\circ f_{6})z-(az-a\circ f_{1})(f_{1}-z)f_{6}. $$
So, if $z_{0}=\nu\lambda\delta z$ then $az_{0}\in D=\sum\limits_{i=0}^{7} Z\cdot e_{i}.$ As $Bz_{0}\simeq B$, the superalgebra $B$ embeds into the free finitely generated $Z$-module $D$.

\medskip

1b. $\mu = 0.$ Then
$$\lambda\delta^{2}z^{2}\alpha_{0}=((((f_{4}-f_{5})(az-a\circ f_{1}))f_{4})\circ f_{6})f_{6}. $$
It means that $\lambda\delta^{2}z^{2}\alpha_{0}\in Z.$ Therefore,
$$\lambda^{2}\delta^{4}\alpha_{1}z=((\lambda\delta^{2} az^{2} - \lambda\delta^{2} za\circ f_{1} - \lambda\delta^{2}\alpha_{0}z^{2})\circ f_{6})f_{6} ,$$
hence $\lambda^{2}\delta^{4}\alpha_{1}z\in Z$.
$$\lambda^{3}\delta^{4}\alpha_{2}z=\lambda^{2}\delta^{4}z a\circ f_{3} - \lambda^{2}\delta^{4}z\alpha_{1}f_{3} $$
$$\lambda^{3}\delta^{4} z\alpha_{3}=\lambda^{2}\delta^{2}z a\circ f_{2} - \lambda^{2}\delta^{4}z\alpha_{1}f_{2} $$
$$\lambda^{2}\delta^{5}\alpha_{4}z=\lambda^{2}\delta^{4}z a\circ f_{5} - \lambda^{2}\delta^{4}z\alpha_{1}f_{5} $$
$$\lambda^{2}\delta^{5}\alpha_{5}z=\lambda^{2}\delta^{4}z a\circ f_{4} - \lambda^{2}\delta^{4}z\alpha_{1}f_{4} $$
$$\lambda^{3}\delta^{5}\alpha_{6}=\lambda^{2}\delta^{4}z a\circ f_{7} - \lambda^{2}\delta^{4}z\alpha_{1}f_{7} $$
$$\lambda^{3}\delta^{5}\alpha_{7}=\lambda^{2}\delta^{4}z a\circ f_{6} - \lambda^{2}\delta^{4}z\alpha_{1}f_{6} $$
As $\lambda^{2}\delta^{4}z\alpha_{1}\in Z,$ we have $\lambda^{3}\delta^{5}\alpha_{i}z^{2}\in Z.$ So if $z_{0}=\lambda^{3}\delta^{5}z^{2}$ then $Bz_{0}$ embeds into $D= \sum\limits_{i=0}^{7} Z\cdot e_{i}.$ As $Bz_{0}\simeq B$, the superalgebra $B$ embeds into the free finitely generated $Z$-module $D$.

\medskip

2. $M=0$. Then $B=A$ is a prime nondegenerate alternative algebra, so it is a central order in the Cayley--Dickson algebra by \cite{Slater}. Using the proof in \cite{Panasenko} and case 1 we have that $B$ embeds into a free finitely generated $Z$-module.

\medskip

3. $char F = 3,$ $\overline{A}=k\cdot 1,$ $\overline{M}=k\cdot x + k\cdot y,$ $1$ is the unit, $x^{2}=y^{2}=0$ and $xy=-yx=1.$ There exists $z\in Z$, $z\neq 0$, such that $x=\frac{x_{0}}{z},$ $y=\frac{y_{0}}{z}.$ Let $a\in B.$ Then $a=\alpha x + \beta y + \gamma\cdot 1$ and
$$((ay_0)y_0)x_0 = -\alpha z^3,$$
$$((ax_0)x_0)y_0 = -\beta z^3,$$
$$(((ay_0)x_0)x_0)y_0=-\gamma z^4.$$
So, if $z_{0}=z^{4}$ then $az_{0}\in D=Z\cdot x + Z\cdot y + Z\cdot 1.$ As $Bz_{0}\simeq B$, hence $B$ embeds into the free finitely generated $Z$-module $D$.

\medskip

4. $char F = 3,$ $\overline{A}=M_{2}(k)$ is the matrix algebra over $k,$ $\overline{M}=k\cdot m_{1} + k\cdot m_{2},$ 
$$m_{1}^{2}=-e_{21}, m_{2}^{2}=e_{12}, m_1m_2=e_{11}, m_2m_1=-e_{22},$$
$$e_{i j}\cdot m_{k}=\delta_{ik}m_{j},$$  
$$m\cdot a = \overline{a}m, \text{ if } a\in A,$$
where $\overline{a}$ is the symplectic involution on $M_{2}(k).$ There exists $z\in Z$, $z\neq 0$, such that $e_{ij}=\frac{f_{ij}}{z},$ $m_{i}=\frac{n_{i}}{z}.$ 
Let $a\in B.$ Then $a=\sum\limits_{i,j}\alpha_{i j}e_{ij} + \beta_{1}m_{1} + \beta_{2}m_{2}$ and the following relations hold:
$$(f_{12}(((an_{1})n_{1})f_{12}))f_{21} + z^{2}((an_{1})n_{1})f_{12} = -\alpha_{1 1}z^{5},$$
$$f_{12}(((an_{1})n_{1})f_{21}) + (((an_{1})n_{1})f_{21})f_{12} = -\alpha_{1 2}z^{4},$$
$$f_{21}(((an_{2})n_{2})f_{12}) + (((an_{2})n_{2})f_{12})f_{21} = \alpha_{2 1}z^{4},$$
$$(f_{21}(((an_{2})n_{2})f_{21}))f_{12} + z^{2}((an_{2})n_{2})f_{21} = \alpha_{2 2}z^{5},$$
$$((f_{22}(f_{12}a))n_{2})f_{21}+f_{21}((f_{22}(f_{12}a))n_{2})=\beta_{1}z^{4},$$
$$((f_{11}(f_{21}a))n_{1})f_{12}+f_{12}((f_{11}(f_{21}a))n_{1})=-\beta_{2}z^{4}.$$
So, if $z_{0}=z^{5}$ then $az_{0}\in D=\sum\limits_{i,j} Z\cdot e_{ij} + Z\cdot m_{1} + Z\cdot m_{2}.$ As $Bz_{0}\simeq B$, hence $B$ embeds into the free finitely generated $Z$-module $D$.

\medskip

5. $char F = 2,$ $k[u]= k + ku,$ $u^{2}=\alpha\neq 0\in k$ is a two-dimensional superalgebra and $\overline{B}=k[u]\otimes_{k}\mathbb{O}=\mathbb{O}+\mathbb{O}u.$ Consider the same basis of $\mathbb{O}$ as in case 1.

There exists $z\in Z$, $z\neq 0$, such that $e_{i}=\frac{f_{i}}{z},
\mu=\frac{\nu}{z}, \beta=\frac{\lambda}{z},
\gamma=\frac{\delta}{z}, \alpha=\frac{\varepsilon}{z}, u=\frac{v}{z},$ where  $f_{i}\in B$, $\delta,\lambda,\alpha\in
Z\setminus\{0\}$, $\nu\in Z,$ $v\in Z(B)\cap M$. Let $a\in B.$ Then $a=\sum\limits_{i=0}^{7}\alpha_{i}e_{i} + \sum\limits_{i=0}^{7}\beta_{i}e_{i}u.$ We have two cases:

5a. $\mu \neq 0.$ Then
$$\nu z(\alpha_{1}+\beta_{1}u)=(az-a\circ f_{1})(f_{1}-z),$$
$$z\nu\lambda(\alpha_{2}+\beta_{2}u)=\nu(a\circ f_{3})z -(az-a\circ f_{1})(f_{1}-z)f_{3}, $$
$$z\nu\lambda(\alpha_{3}+\beta_{3}u)=\nu(a\circ f_{2})z-(az-a\circ f_{1})(f_{1}-z)f_{2}, $$
$$z\nu\delta(\alpha_{4}+\beta_{4}u)=\nu(a\circ f_{5})z-(az-a\circ f_{1})(f_{1}-z)f_{5}, $$
$$z\nu\delta(\alpha_{5}+\beta_{5}u)=\nu(a\circ f_{4})z-(az-a\circ f_{1})(f_{1}-z)f_{4}, $$
$$\nu\lambda\delta(\alpha_{6}+\beta_{6}u)=\nu(a\circ f_{7})z-(az-a\circ f_{1})(f_{1}-z)f_{7}, $$
$$\nu\lambda\delta(\alpha_{7}+\beta_{7}u)=\nu(a\circ f_{6})z-(az-a\circ f_{1})(f_{1}-z)f_{6}.$$
So, if $z_{0}=\nu\lambda\delta z$ then $az_{0}\in D_{1}=\sum\limits_{i=0}^{7} (B\cap(k[u]))\cdot e_{i}.$ 

\medskip

5b. $\mu = 0.$ 
Then
$$\lambda\delta^{2}z^{2}(\alpha_{0}+\beta_{0}u)=((((f_{4}-f_{5})(az-a\circ f_{1}))f_{4})\circ f_{6})f_{6}. $$
It means that $\lambda\delta^{2}z^{2}(\alpha_{0}+\beta_{0}u)\in (B\cap k[u]).$ We also have
$$\lambda^{2}\delta^{4}(\alpha_{1}+\beta_{1}u)z=((\lambda\delta^{2} az^{2} - \lambda\delta^{2} za\circ f_{1} - \lambda\delta^{2}(\alpha_{0}+\beta_{0}u) z^{2})\circ f_{6})f_{6} ,$$
hence $\lambda^{2}\delta^{4}(\alpha_{1}+\beta_{1}u)z\in Z$.
As above, the following relations hold:
$$\lambda^{3}\delta^{4}\alpha_{2}z=\lambda^{2}\delta^{4}z a\circ f_{3} - \lambda^{2}\delta^{4}z(\alpha_{1}+\beta_{1}u)f_{3},$$
$$\lambda^{3}\delta^{4} z\alpha_{3}=\lambda^{2}\delta^{2}z a\circ f_{2} - \lambda^{2}\delta^{4}z(\alpha_{1}+\beta_{1}u)f_{2}, $$
$$\lambda^{2}\delta^{5}\alpha_{4}z=\lambda^{2}\delta^{4}z a\circ f_{5} - \lambda^{2}\delta^{4}z(\alpha_{1}+\beta_{1}u)f_{5}, $$
$$\lambda^{2}\delta^{5}\alpha_{5}z=\lambda^{2}\delta^{4}z a\circ f_{4} - \lambda^{2}\delta^{4}z(\alpha_{1}+\beta_{1}u)f_{4}, $$
$$\lambda^{3}\delta^{5}\alpha_{6}=\lambda^{2}\delta^{4}z a\circ f_{7} - \lambda^{2}\delta^{4}z(\alpha_{1}+\beta_{1}u)f_{7}, $$
$$\lambda^{3}\delta^{5}\alpha_{7}=\lambda^{2}\delta^{4}z a\circ f_{6} - \lambda^{2}\delta^{4}z(\alpha_{1}+\beta_{1}u)f_{6}. $$
As $\lambda^{2}\delta^{4}z(\alpha_{1}+\beta_{1}u)\in Z,$ then $\lambda^{3}\delta^{5}(\alpha_{i}+\beta_{i}u)z^{2}\in Z.$ So if $z_{0}=\lambda^{3}\delta^{5}z^{2}$ then $Bz_{0}$ embeds into $D_{1}= \sum\limits_{i=0}^{7} (B\cap k[u])\cdot e_{i}.$

In both cases we have $Bz_{0}\subset \sum\limits_{i=0}^{7} (B\cap k[u])e_i.$ If $a\in A$ then  $az_{0}\in \sum\limits_{i=0}^7(B\cap \nolinebreak[4]k)e_i=\sum\limits_{i=0}^7 Ze_i$ and $Az_{0}\subset \sum\limits_{i=0}^{7} Z\cdot e_{i}.$ If $a\in M$ then $az_{0}v\in\sum\limits_{i=0}^7(B\cap k)e_i = \sum\limits_{i=0}^7 Ze_i$ and $Mz_{0}v\in\sum\limits_{i=0}^{7} Z\cdot e_{i}.$ So, $Az_{0}$ and $Mz_{0}v$ are finite $Z$-modules, $Az_{0}\simeq A$ and $Mz_{0}v\simeq M$, hence $B=A+M$ is a finite $Z$-module. Theorem is proved.

\end{proof}

In \cite{Shest}, I.P. Shestakov also classified all prime alternative superalgebras under some restriction.

\begin{definition}
	Let $A$ be a $Z_2$-graded algebra. It is called \textbf{prime}, if for every two nonzero ideals $I$ and $J$ the statement $IJ=0$ implies $I=0$ or $J=0$.
\end{definition}

\noindent\textbf{Theorem} (I.P. Shestakov, \cite{Shest}). {\it Let $B=A+M$ be a prime alternative nonassociative superalgebra over a field $k$, $M\neq 0$. Let one of the following two statements hold:

1) there exist $a,b\in A$ such that $(ab-ba)^4\neq 0$;

2) $A$ does not contain nonzero nil-ideals $I$ satisfying  condition $(I,M,M)\subset I$.

\noindent Then the even part $Z=Z(B)_0$ of a center of superalgebra $B$ is not zero, $Z$ does not have zero divisors of $B$, and $Z^{-1}B$ is a central simple alternative superalgebra over the field $Z^{-1}Z$.}

So we trivially have the following corollary. 

\begin{corollary}
 Let $B$ be a unital prime alternative nonassociative superalgebra with a restriction 1) or 2) from Shestakov's Theorem and $Z=Z(B)_{0}$. Then either $B$ is an order in $B(\Gamma, d, \gamma)$ or $B$ embeds into a finitely generated $Z$-module. 
\end{corollary}

\section{Central orders in classical simple Jordan superalgebras with a semi-simple even part.}

In this section all fields are assumed to be of characteristic not equal 2. 

If $A=A_0 + A_1$ is a $Z_2$-graded algebra, then we may define an algebra $A^{+}$ with the same linear structure, but with a new multiplication: $$x\circ_s y=\frac{1}{2}(xy+(-1)^{|x||y|}yx)$$ for every $x,y\in A_0\cup A_1$.  If $A$ is an associative superalgebra then $A^{+}$ is a Jordan superalgebra. 

Let $A$ be an associative superalgebra with a superinvolution $x\rightarrow x^*$. Then the algebra of symmetric elements relative to this superinvolution is a subsuperalgebra of $A^+$ and denotes by $H(A,*)$. Let the notation $H_n(A)$ stand for the subsuperalgebra of symmetric elements of $M_n(A)$ with a natural grading relative to the $*$-transpose involution.

In \cite{RacineZelm} M.L. Racine and E.I. Zel'manov classified all finite dimensional simple Jordan superalgebras with a semi-simple even part. Let us give examples of simple Jordan superalgebras.

\noindent\textbf{Example 3.} a) If $B=M_{n+m}(k)$ then $B^{+}$ is denoted by $M_{n,m}(k)$. 

b) We will use the following notation: $$Q_n(k)=\left\{\begin{pmatrix}
a & b \\
b & a 
\end{pmatrix}\mid a,b\in M_n(k)\right\}^{+}.$$
Then $Q_n(k)$ is a simple Jordan superalgebra. It is easy to see that $Q_n(k)=M_n(k)+\overline{M_n(k)}=A+M$, where $A=M_n(k)$ acts on $M=\overline{M_n(k)}$ like the Jordan matrix multiplication on $M_n(k)$ and for $\overline{x},\overline{y}\in\overline{M_n(k)}$ we have $\overline{x}\overline{y}=\frac{1}{2}(x\cdot y-y\cdot x)$, where $x\cdot y$ is the multiplication in the associative algebra $M_n(k)$.

c) We will use the following notation: $$P_n(k)=\left\{\begin{pmatrix}
a & b \\
c & d
\end{pmatrix}\in M_{n+n}(k)\mid a^T=d, b^T=-b, c^T=c\right\}.$$
Then $P_n(k)$ is a simple Jordan superalgebra, if $n>1$. Let the notation $S_n(k)$ stand for the subspace of skewsymmetric elements of $M_n(k)$. It is easy to see that $P_n(k)=M_n(k)+(\overline{S_n(k)}+\overline{H_n(k)})=A+M$, where $A=M_n(k)$ acts on $M=\overline{S_n(k)}+\overline{H_n(k)}$ in the following way: 
$$a\overline{s}=\frac{1}{2}(as+sa^T), \quad a\in M_n(k), s\in S_n(k),$$
$$a\overline{h}=\frac{1}{2}(a^Th+ha), \quad a\in M_n(k), h\in H_n(k).$$
For $\overline{x}_1,\overline{x}_2\in\overline{S_n(k)}$, $\overline{y}_1,\overline{y}_2\in \overline{H_n(k)}$ we have $$\overline{x}_1\overline{x}_2=\overline{y}_1\overline{y}_2=0,$$ $$\overline{x}_1\overline{y}_1=\frac{1}{2}(x_1\cdot y_1-y_1\cdot x_1),$$ where $x\cdot y$ is the multiplication in $M_n(k)$.

d) We will use the following notation:
$$osp_{n,2m}(k)=\left\{\begin{pmatrix}
a & b_1 & b_2 \\
-b_2^T & c_1 & c_2 \\
b_1^T & c_3 & c_1^T 
\end{pmatrix}\mid a=a^T, c_2=-c_2^T, c_3=-c_3^T\right\},$$
where $a\in M_n(k)$, $b_i\in M_{n\times m}(k)$ (matrices with $n$ lines and $m$ rows), $d_i\in M_m(k)$.

Then $osp_{n,2m}(k)$ is a simple Jordan superalgebra with $Z_2$-grading $$\begin{pmatrix}
a & b_1 & b_2 \\
-b_2^T & c_1 & c_2 \\
b_1^T & c_3 & c_1^T 
\end{pmatrix}= \begin{pmatrix}
a & 0 & 0 \\
0 & c_1 & c_2 \\
0 & c_3 & c_1^T 
\end{pmatrix}+ \begin{pmatrix}
0 & b_1 & b_2 \\
-b_2^T & 0 & 0 \\
b_1^T & 0 & 0 
\end{pmatrix},$$
i.e. $osp_{n,2m}(k)$ is a subsuperalgebra of $M_{n,2m}(k)$.

e) Let $A=ke_1+ke_2$, $M=kx+ky$, where $e_1,e_2$ are orthogonal idempotents, $e_ix=\frac{1}{2}x$, $e_iy=\frac{1}{2}y$, $xy=e_1+te_2$, where $0\neq t\in k$. Then $D_t=A+M$ is a Jordan simple superalgebra. 

f) Let $V=V_0\oplus V_1$ be a $Z_2$-graded vector space over $k$ and let $(\cdot,\cdot)$ be a non-degenerate superform on $V$ such that its restriction on $V_0\times V_0$ is symmetric, but its restriction on $V_1\times V_1$ is skew-symmetric. Let $J=A+M$, where $A=V_0+ke$, $M=V_1$, where $e$ is a unit of $J$ and $xy=(x,y)e$ for $x,y\in V$. Then $J$ is a simple Jordan superalgebra.

g) Let $A=H_3(k)$. Then there is a Jordan action of $H_3(k)$ on $S_3(k)$. We will use notations $\overline{S_3(k)}$ and $\overline{\overline{S_3(k)}}$ for two isomorphic copies of $H_3(k)$-modules $S_3(k)$, so $M=\overline{S_3(k)}\oplus\overline{\overline{S_3(k)}}$. Let us define a multiplication on $J=A+M$. If $\overline{x}_1,\overline{x}_2\in \overline{S_3(k)}$ and $\overline{\overline{y}}_1,\overline{\overline{y}}_2\in \overline{\overline{S_3(k)}}$ then $$\overline{x}_1\overline{x}_2=\overline{\overline{y}}_1\overline{\overline{y}}_2=0,$$ $$\overline{x}_1\overline{\overline{y}}_1=x_1\cdot y_1,$$
where $x_1\cdot y_1$ is the multiplication in the Jordan algebra $M_3(k)^+$. Then $J=H_3(k)+(\overline{S_3(k)}\oplus \overline{\overline{S_3(k)}})$ is a simple Jordan superalgebra.

h) Let $B(4,2)=A+M$ from Example 2.d. There is an involution $*$ on $B(4,2)$, $(a+m)^*=\overline{a}-m$, where $a\in A$, $m\in M$ and $\overline{a}$ is the symplectic involution on $A=M_2(k)$. Then define a superinvolution on $M_3(B(4,2))$, which is the ${}^*$-transpose superinvolution. The set of symmetric elements with respect to this superinvolution is denoted by $H_3(B(4,2))$, it is a simple Jordan superalgebra. 

\begin{definition}
	We will call algebras from Example 3 as \textbf{classical simple Jordan superalgebras}.
\end{definition}

\noindent\textbf{Theorem (M.L. Racine, E.I. Zel'manov, \cite{RacineZelm})}.
\textit{Every unital simple finite dimensional Jordan superalgebra with a semisimple even part over an algebraically closed field is isomorphic to one of the classical simple Jordan superalgebras.}

\begin{theorem}
	Let $J=A+M$ be a classical simple Jordan superalgebra and $Z=Z(J)_0$. Then a unital central order in $J$ embeds into a free finitely generated $Z$-module. 
\end{theorem}
\begin{proof} Let $B=B_0+B_1$ be a unital Jordan superalgebra,  $Z=Z(B)_{0}$, $k=Z^{-1}Z$, $J=Z^{-1}B$, $A=Z^{-1}B_0$, $M=Z^{-1}B_1$ and let $J=A+M$ be a classical simple Jordan superalgebra over a field $k$. We will consider all posible cases according to Example 3. 
	
a) $J=M_{n,m}(k)$. Then $J$ has a basis $e_{ij}$, $1\le i\le n+m$, $1\le j \le n+m$. If $a\in B$ then $$a=\sum\limits_{i=1}^{n+m}\sum\limits_{j=1}^{n+m} \alpha_{ij}e_{ij}.$$
There exists an element $z\in Z$, $z\neq 0$, such that $e_{ij}=\frac{f_{ij}}{z}$ and $f_{ij}\in B$. 

We have $$2(a\circ_s f_{kk})\circ_s f_{kk}-a\circ_s f_{kk}=z^2\circ_s\alpha_{kk}e_{kk}$$
Let $i\neq k$. Then we have $$z^2\alpha_{kk}(e_{kk}\circ_sf_{ik})\circ_sf_{ki}=\frac{1}{4}z^4\alpha_{kk}(e_{ii}\pm e_{kk})\in B.$$
So, we have that $z^4\alpha_{kk}e_{ii}\in B$ for $1\le i \le n+m$ and $$z^4\alpha_{kk}=\sum\limits_{i=1}^{n+m}z^4\alpha_{kk}e_{ii}\in B\cap Z(J)_0=Z.$$ 
In other hand, $$\frac{1}{8}\alpha_{lk}(e_{ll}\pm e_{kk})=((a\circ_sf_{kk})\circ_sf_{ll})\circ_sf_{kl}\in B.$$
Hence $$z^4\alpha_{lk}e_{ll}=16(((a\circ_sf_{kk})\circ_sf_{ll})\circ_sf_{kl})\circ_sf_{ll}\in B$$ 
for $1\le l\le n+m$. Then for $1\le i\le n+m$ we have $$z^6\alpha_{lk}(e_{ii}\pm e_{ll})=4z^4\alpha_{lk}(e_{ll}\circ_sf_{il})\circ_sf_{li}\in B.$$
Then $z^6\alpha_{lk}e_{ii}\in B$  and $z^6\alpha_{lk}=\sum\limits_{i=1}^{n+m}z^6\alpha_{lk}e_{ii}\in B\cap Z(J)_0=Z$. Hence, we have $z^6a\in \sum_{i,j=1}^{n+m}Ze_{ij}$. So the $Z$-module $z^6B\simeq B$ embeds into the free finitely generated $Z$-module $\sum_{i,j=1}^{n+m}Ze_{ij}$.
\medskip

b) $J=Q_n(k)$.  Then $J$ has a basis $e_{ij}, \overline{e}_{ij}$, $1\le i,j\le n$. If $a\in B$ then $$a=\sum\limits_{i,j=1}^n \alpha_{ij}e_{ij}+\sum\limits_{i,j=1}^n \beta_{ij}\overline{e}_{ij}.$$
There exists an element $z\in Z$, $z\neq 0$, such that $e_{ij}=\frac{f_{ij}}{z}$, $\overline{e}_{ij}=\frac{\overline{f}_{ij}}{z}$ and $f_{ij},\overline{f}_{ij}\in B$. As in the case (a),
$$z^6(\alpha_{ij}\cdot 1+\beta_{ij}\cdot \overline{1}) \in B\cap Z(J).$$
Hence $z^6\alpha_{ij}\in B\cap Z(J)_0=Z$, $z^6\beta_{ij}\cdot\overline{1}\in B$. But 
\begin{multline*}z^{11}\beta_{ij}=\sum_{k=2}^n8(((((z^6\beta_{ij}\cdot \overline{1})\circ_sf_{12})\circ_s\overline{f}_{21})\circ_sf_{11})\circ_sf_{1k})\circ_sf_{k1} -\\- 2(n-2)z^2(((z^6\beta_{ij}\cdot\overline{1})\circ_sf_{12})\circ_s\overline{f}_{21})\circ_sf_{11}.\end{multline*}

So $z^{11}\beta_{ij}\in Z$. If $z_0=z^{11}$ then the $Z$-module $z_0B\simeq B$ embeds into the free finitely generated $Z$-module $\sum\limits_{i,j=1}^n Ze_{ij}+\sum\limits_{i,j=1}^n \overline{e}_{ij}$. 
\medskip

c) $J=P_n(k)$. Then $J$ has a basis $e_{ij}\in M_n(k)$, $1\le i,j\le n$, $\overline{e}_{ij}\in \overline{H_n(k)}$, $1\le i \le j \le n$, $\overline{\overline{e}}_{ij}\in \overline{S_n(k)}$, $1\le i<j\le n$.  If $a\in B$ then $$a=\sum\limits_{i,j=1}^n \alpha_{ij}e_{ij}+\sum\limits_{1\le i\le j\le n}\beta_{ij}\overline{e}_{ij}+\sum\limits_{1\le i< j\le n}\gamma_{ij}\overline{\overline{e}}_{ij}.$$
There exists an element $z\in Z$, $z\neq 0$, such that $e_{ij}=\frac{f_{ij}}{z}$, $\overline{e}_{ij}=\frac{\overline{f}_{ij}}{z}$, $\overline{\overline{e}}_{ij}=\frac{\overline{\overline{f}}_{ij}}{z}$ and $f_{ij},\overline{f}_{ij}\in B$. In the same way as in the case (a), one may show that $z^6\alpha_{ij}+b\in B$, where $b$ is a some element from $M$. Then $z^6\alpha_{ij}\in Z$. 

Let $Y=\begin{pmatrix}
a & s \\
h & a^T
\end{pmatrix}\in J$, i.e. $a\in M_n(k)$, $h\in H_n(k)$, $s\in S_n(k)$. We have the following identity:
$$Y\begin{pmatrix}
0 & 0 \\
E & 0
\end{pmatrix}=\frac{1}{2}\begin{pmatrix}
s & d \\
x & -s
\end{pmatrix}$$
for some $d\in\overline{S_n(k)}$, $x\in\overline{H_n(k)}$. Hence, like with $\alpha_{ij}$ we have $z^7\gamma_{ij}\in Z$ (since $\begin{pmatrix}
0 & 0\\
E & 0
\end{pmatrix}=\frac{1}{z}(f_{11}+\dots+f_{nn})$).

Using a multiplication by 
$\begin{pmatrix}
0 & f_{1i}-f_{i1} \\
0 & 0
\end{pmatrix}$, we take a matrix with $z\beta_{kj}$ on some place in the upper left quarter. So, we have $z^{7}\beta_{kj}\in Z$. If $z_0=z^7$ then the $Z$-module $z_0B\simeq B$ embeds into the free finitely generated $Z$-module $\sum\limits_{i,j=1}^n Ze_{ij}+\sum\limits_{1\le i\le j\le n} \overline{e}_{ij}+\sum\limits_{1\le i< j\le n} \overline{\overline{e}}_{ij}$.

\medskip

d) $J=osp_{n,2m}(k)$. If $x\in B$ then $$x=\begin{pmatrix}
a & b_1 & b_2 \\
-b_2^T & c_1 & c_2 \\
b_1^T & c_3 & c_1^T 
\end{pmatrix}.$$
We will use the notation $$S=\begin{pmatrix}
E & 0 & 0 \\
0 & 0 & 0 \\
0 & 0 & 0
\end{pmatrix}.$$
Then
$$\begin{pmatrix}
a & 0 & 0 \\
0 & 0 & 0 \\
0 & 0 & 0
\end{pmatrix}=(2x\circ_sS - x)\circ_sS,$$
$$\begin{pmatrix}
0 & 0 & 0 \\
0 & c_1 & c_2 \\
0 & c_3 & c_1^T
\end{pmatrix}=(2x\circ_sS-x)\circ_s(1-S).$$
The first equation allows us to use case (a). The second one allows us to use case (c). Let $$b_1=\sum\limits_{i=1}^n\sum\limits_{j=1}^m\beta_{ij}e_{ij},$$ $$b_2=\sum\limits_{i=1}^n\sum\limits_{j=1}^m\gamma_{ij}e_{ij}.$$
Then we have:\small
 \begin{multline*}
\begin{pmatrix}
0 & 0 & 0 \\
0 & 0 & \gamma_{ij}e_{kj}-\gamma_{ij}e_{jk} \\
0 & 0 & 0
\end{pmatrix}=\\=2\Bigg(\left(\left(((x\circ_sS - x)\circ_sS)\circ_s\begin{pmatrix}
e_{ii} & 0 & 0 \\
0 & 0 & 0 \\
0 & 0 & 0
\end{pmatrix}\right)\circ_s\begin{pmatrix}
0 & 0 & 0\\
0 & e_{jj} & 0 \\
0 & 0 & e_{jj}
\end{pmatrix}\right)\circ_s\\
\circ_s\begin{pmatrix}
0 & 0 & 0\\
0 & e_{rj} & 0 \\
0 & 0 & e_{jr}
\end{pmatrix}\Bigg)\circ_s\begin{pmatrix}
0 & 0 & e_{rk}\\
-e_{kr} & 0 & 0 \\
0 & 0 & 0
\end{pmatrix} ,\end{multline*}\normalsize
where $r\neq i$, $k\neq j$. 

This identity allows us to use case (c). Similarly, we can use it for $\beta_{ij}$. So, we can find $z\in Z$ such that $zB\simeq B$ embeds into a free finitely generated $Z$-module.

\medskip

e) $J=D_t$, $0\neq t\in k$. Then $J$ has a basis $e_1,e_2,x,y$. If $a\in B$ then $$a=\alpha_1 e_1+\alpha_2e_2+\beta_1x+\beta_2y.$$
There exists an element $z\in Z$, $z\neq 0$, such that $e_i=\frac{f_i}{z}$, $x=\frac{v}{z}$, $y=\frac{w}{z}$, $t=\frac{u}{z}$ and $f_i,v,w\in B$, $0\neq u\in Z$. We have the following identities:
$$uz^4\beta_1=4(((af_1)f_2)w)(uf_1+zf_2),$$
$$uz^4\beta_2=-4(((af_1)f_2)v)(uf_1+zf_2),$$
$$uz^6\alpha_1=4(((((af_1)v)f_1)f_2)w)(uf_1+zf_2),$$
$$uz^6\alpha_2=4(((((af_2)v)f_1)f_2)w)(uf_1+zf_2).$$
If $z_0=uz^6$ then $z_0B\simeq B$ embeds into the free finitely generated $Z$-module $Ze_1+Ze_2+Zx+Zy$. 

\medskip

f) $J$ is the superalgebra of non-degenerate supersymmetric superform on a superspace $V=V_0\oplus V_1$. Then $J$ has a basis $e,e_1,\dots,e_n, g_1,\dots,g_{2m}$, where $e$ is a unit, $\{e_1,\dots, e_n\}$ is an orthogonal basis of $V_0$, $(e_i,e_i)=\alpha_i$, $\{g_1,\dots,g_{2m}\}$ is a basis of $V_1$, $(g_{2i-1},g_{2i})=\beta_i\neq 0$ and $(g_i,g_j)=0$ otherwise. If $a\in B$ then $$a=\gamma e + \sum\limits_{i=1}^{n}\delta_i e_i + \sum\limits_{i=1}^{2k}\varepsilon_ig_i.$$
There exists an element $z\in Z$, $z\neq 0$, such that $e_i=\frac{f_i}{z}$, $g_i=\frac{h_i}{z}$, $\alpha_i=\frac{\lambda_i}{z}$, $\beta_i=\frac{\mu_i}{z}$ and $f_i,h_i\in B$, $\lambda_i,\mu_i\in Z$. We have the following identities:
$$z^2\lambda_1\lambda_2\gamma=(((af_1)f_1)f_2)f_2,$$
$$z\lambda_1\lambda_i\delta_i=((af_i)f_1)f_1, \quad i>1,$$
$$z\lambda_1\lambda_2\delta_1=((af_1)f_2)f_2.$$
Therefore,
$$((ah_{2k})h_{2k})h_{2k-1}=-z\mu_k^2\varepsilon_{2k-1},$$
$$((ah_{2k-1})h_{2k-1})h_{2k}=-z\mu_k^2\varepsilon_{2k}.$$
If $z_0=z^2\prod_{i=1}^{n}\lambda_i\prod_{i=1}^{m}\mu_i^2$ then $z_0B\simeq B$ embeds into the free finitely generated $Z$-module $Ze+\sum\limits_{i=1}^n Ze_i + \sum\limits_{i=1}^{2m}Zg_i$. 
\medskip

g) $J=H_3(k)+(\overline{S_3(k)}\oplus \overline{\overline{S_3(k)}})$. Then $J$ has a basis $e_{11},e_{22},e_{33}$, $e_{12},e_{13},e_{23}$, $\overline{e}_{12},\overline{e}_{13},\overline{e}_{23}$, $\overline{\overline{e}}_{12},\overline{\overline{e}}_{13},\overline{\overline{e}}_{23}.$ If $a\in B$ then $$a=\sum\limits_{1\le i\le j\le 3}\alpha_{ij}e_{ij}+\sum\limits_{1\le i<j\le 3}\beta_{ij}\overline{e}_{ij}+\sum\limits_{1\le i<j\le 3}\gamma_{ij}\overline{\overline{e}}_{ij}.$$
There exists an element $z\in Z$, $z\neq 0$, such that $e_{ij}=\frac{f_{ij}}{z}$, $\overline{e}_{ij}=\frac{\overline{f}_{ij}}{z}$, $\overline{\overline{e}}_{ij}=\frac{\overline{\overline{f}}_{ij}}{z}$ and $f_{ij}, \overline{f}_{ij}, \overline{\overline{f}}_{ij}\in B$. Let $$F_1(a)=((((a\overline{\overline{f}}_{12})\overline{\overline{f}}_{12})\overline{f}_{12})f_{11})f_{33},$$ 
$$F_2(a)=((((a\overline{\overline{f}}_{12})\overline{\overline{f}}_{12})\overline{f}_{12})f_{22})f_{33},$$
$$F_3(a)=((((a\overline{\overline{f}}_{13})\overline{\overline{f}}_{13})\overline{f}_{13})f_{33})f_{22},$$
$$G_1(a)=((((a\overline{f}_{12})\overline{f}_{12})\overline{\overline{f}}_{12})f_{11})f_{33},$$ 
$$G_2(a)=((((a\overline{f}_{12})\overline{f}_{12})\overline{\overline{f}}_{12})f_{22})f_{33},$$
$$G_3(a)=((((a\overline{f}_{13})\overline{f}_{13})\overline{\overline{f}}_{13})f_{33})f_{22}.$$
 Then we have the following identities:
$$z^7\beta_{23}=16zF_1(a)f_{13}+32(F_1(a)f_{23})f_{12}+32(F_1(a)f_{12})f_{23},$$
$$z^7\beta_{13}=16zF_2(a)f_{23}+32(F_2(a)f_{13})f_{12}+32(F_2(a)f_{12})f_{13},$$
$$z^7\beta_{12}=16zF_3(a)f_{23}+32(F_3(a)f_{12})f_{13}+32(F_3(a)f_{13})f_{12},$$
$$z^7\gamma_{23}=16zG_1(a)f_{13}+32(G_1(a)f_{23})f_{12}+32(G_1(a)f_{12})f_{23},$$
$$z^7\gamma_{13}=16zG_2(a)f_{23}+32(G_2(a)f_{13})f_{12}+32(G_2(a)f_{12})f_{13},$$
$$z^7\gamma_{12}=16zG_3(a)f_{23}+32(G_3(a)f_{12})f_{13}+32(G_3(a)f_{13})f_{12},$$
$$z^4\alpha_{12}=2z((af_{11})f_{22})f_{12}+4(((af_{11})f_{22})f_{13})f_{23}+4(((af_{11})f_{22})f_{23})f_{13},$$
$$z^4\alpha_{13}=2z((af_{11})f_{33})f_{13}+4(((af_{11})f_{33})f_{12})f_{23}+4(((af_{11})f_{33})f_{23})f_{12},$$
$$z^4\alpha_{23}=2z((af_{22})f_{33})f_{23}+4(((af_{22})f_{33})f_{12})f_{13}+4(((af_{22})f_{33})f_{13})f_{12},$$
$$z^4\alpha_{11}=2(((2af_{11}-za)f_{11})f_{12})f_{12}+2(((2af_{11}-za)f_{11})f_{13})f_{13}-z^2(2af_{11}-za)f_{11},$$
$$z^4\alpha_{22}=2(((2af_{22}-za)f_{22})f_{12})f_{12}+2(((2af_{22}-za)f_{22})f_{23})f_{23}-z^2(2af_{22}-za)f_{22}$$
$$z^4\alpha_{33}=2(((2af_{33}-za)f_{33})f_{23})f_{23}+2(((2af_{33}-za)f_{33})f_{13})f_{13}-z^2(2af_{33}-za)f_{33}.$$
If $z_0=z^4$ then $z_0B\simeq B$ embeds into the free finitely generated $Z$-module $\sum\limits_{1\le i\le j\le 3} Ze_{ij}+\sum\limits_{1\le i<j\le 3} Z\overline{e}_i + \sum\limits_{1\le i<j\le 3}Z\overline{\overline{e}}_{ij}$. 

\medskip

h) $J=H_3(B(4,2))$. If $a\in B$ then $$a=\begin{pmatrix}
\alpha_1 & b_1 & b_2 \\
\overline{b_1} & \alpha_2 & b_3 \\
\overline{b_2} & \overline{b_3} & \alpha_3
\end{pmatrix},$$
where $$b_1=\sum\limits_{i,j=1}^2\alpha_{ij}e_{ij}+\delta_1 m_1+\delta_2 m_2,$$
$$b_2=\sum\limits_{i,j=1}^2\beta_{ij}e_{ij}+\varepsilon_1 m_1+\varepsilon_2 m_2,$$
$$b_3=\sum\limits_{i,j=1}^2\gamma_{ij}e_{ij}+ \mu_1m_1+\mu_2 m_2.$$

Let us denote $$b[12]=\begin{pmatrix}
0 & b & 0 \\
\overline{b} & 0 & 0 \\
0 & 0 & 0
\end{pmatrix}, \quad b[13]=\begin{pmatrix}
0 & 0 & b \\
0 & 0 & 0 \\
\overline{b} & 0 & 0
\end{pmatrix}, \quad b[23]=\begin{pmatrix}
0 & 0 & 0 \\
0 & 0 & b \\
0 & \overline{b} & 0
\end{pmatrix},$$
if $b\in B(4,2)$. Also, $e_{ij}=1[ij]$ as usual.

Then we have the following identities:
$$\alpha_1=4(((2ae_{11}-a)e_{11})e_{12})e_{12}+4(((2ae_{11}-a)e_{11})e_{13})e_{13}-(2ae_{11}-a)e_{11},$$
$$\alpha_2=4(((2ae_{22}-a)e_{22})e_{12})e_{12}+4(((2ae_{22}-a)e_{22})e_{23})e_{23}-(2ae_{22}-a)e_{22},$$
$$\alpha_3=4(((2ae_{33}-a)e_{33})e_{23})e_{23}+4(((2ae_{33}-a)e_{33})e_{13})e_{13}-(2ae_{33}-a)e_{33},$$
$$\delta_1=32((((ae_{11})e_{22})m_2[12])e_{13})e_{13}+8((ae_{11})e_{22})m_2[12]-8(((ae_{11})e_{22})m_2[12])e_{11},$$
$$\delta_2=32((((ae_{11})e_{22})m_1[12])e_{13})e_{13}+8((ae_{11})e_{22})m_1[12]-8(((ae_{11})e_{22})m_1[12])e_{11},$$
$$\varepsilon_1=32((((ae_{11})e_{33})m_2[13])e_{12})e_{12}+8((ae_{11})e_{33})m_2[13]-8(((ae_{11})e_{33})m_2[13])e_{11},$$
$$\varepsilon_2=32((((ae_{11})e_{33})m_1[13])e_{12})e_{12}+8((ae_{11})e_{33})m_1[13]-8(((ae_{11})e_{33})m_1[13])e_{11},$$
$$\mu_1=32((((ae_{22})e_{33})m_2[23])e_{12})e_{12}+8((ae_{22})e_{33})m_2[23]-8(((ae_{22})e_{33})m_2[23])e_{22},$$
$$\mu_2=32((((ae_{22})e_{33})m_1[23])e_{12})e_{12}+8((ae_{22})e_{33})m_1[23]-8(((ae_{22})e_{33})m_1[23])e_{22},$$
$$-\alpha_{21}=32((((ae_{11})e_{22})e_{12}[12])e_{13})e_{13}+8((ae_{11})e_{22})e_{12}[12]-8(((ae_{11})e_{22})e_{12}[12])e_{11},$$
$$-\alpha_{12}=32((((ae_{11})e_{22})e_{21}[12])e_{13})e_{13}+8((ae_{11})e_{22})e_{21}[12]-8(((ae_{11})e_{22})e_{21}[12])e_{11},$$
$$-\beta_{21}=32((((ae_{11})e_{33})e_{12}[13])e_{12})e_{12}+8((ae_{11})e_{33})e_{12}[13]-8(((ae_{11})e_{33})e_{12}[13])e_{11},$$
$$-\beta_{12}=32((((ae_{11})e_{33})e_{21}[13])e_{12})e_{12}+8((ae_{11})e_{33})e_{21}[13]-8(((ae_{11})e_{33})e_{21}[13])e_{11},$$
$$-\gamma_{21}=32((((ae_{22})e_{33})e_{12}[23])e_{12})e_{12}+8((ae_{22})e_{33})e_{12}[23]-8(((ae_{22})e_{33})e_{12}[23])e_{22},$$
$$-\gamma_{12}=32((((ae_{22})e_{33})e_{21}[23])e_{12})e_{12}+8((ae_{22})e_{33})e_{21}[23]-8(((ae_{22})e_{33})e_{21}[23])e_{22},$$
$$\alpha_{11}=32((((ae_{11})e_{22})e_{22}[12])e_{13})e_{13}+8((ae_{11})e_{22})e_{22}[12]-8(((ae_{11})e_{22})e_{22}[12])e_{11},$$
$$\alpha_{22}=32((((ae_{11})e_{22})e_{11}[12])e_{13})e_{13}+8((ae_{11})e_{22})e_{11}[12]-8(((ae_{11})e_{22})e_{11}[12])e_{11},$$
$$\beta_{11}=32((((ae_{11})e_{33})e_{22}[13])e_{12})e_{12}+8((ae_{11})e_{33})e_{22}[13]-8(((ae_{11})e_{33})e_{22}[13])e_{11},$$
$$\beta_{22}=32((((ae_{11})e_{33})e_{11}[13])e_{12})e_{12}+8((ae_{11})e_{33})e_{11}[13]-8(((ae_{11})e_{33})e_{11}[13])e_{11},$$
$$\gamma_{11}=32((((ae_{22})e_{33})e_{22}[23])e_{12})e_{12}+8((ae_{22})e_{33})e_{22}[23]-8(((ae_{22})e_{33})e_{22}[23])e_{22},$$
$$\gamma_{22}=32((((ae_{22})e_{33})e_{11}[23])e_{12})e_{12}+8((ae_{22})e_{33})e_{11}[23]-8(((ae_{22})e_{33})e_{11}[23])e_{22}.$$

We have representations of basis elements of $J$ as fractions with numerators $a_1,a_2,\dots,a_{21}$ from $B$ and the common denominator $z\in Z$. Then our identities imply that $z^5B\simeq B$ embeds into the free finitely generated $Z$-module $\sum\limits_{i=1}^{21} Za_{i}$.		\end{proof}

\section{Acknowlegments} 

 The work is supported by Mathematical Center in Akademgorodok, the agreement with Ministry of Science and High Education of the Russian Federation number  075-15-2019-1613. 
 
 The reported study was funded by RFBR, project number 19-31-90055.
 
 The author is grateful to the referee for the useful comments.
 
  The author also expresses his gratitude to P.~S.~Kolesnikov, whose remarks helped to improve this paper.

\end{document}